\documentclass[12pt]{article}

\usepackage{amsmath}
\usepackage{amsthm}
\usepackage{amssymb}
\usepackage{graphicx}
\usepackage{standalone}
\usepackage{comment}
\usepackage{color}
\usepackage{enumerate}

\usepackage{xcolor}
\usepackage{cancel}

\makeatletter
\usepackage{comment}
\let\wfs@comment@comment\comment
\let\comment\@undefined

\usepackage{changes}
\let\wfs@changes@comment\comment
\let\comment\@undefined

\newcommand\comment{%
    \ifthenelse{\equal{\@currenvir}{comment}}
    {\wfs@comment@comment}
    {\wfs@changes@comment}%
}

\definechangesauthor[name=Corrado, color=red]{COR}
\definechangesauthor[name=Valentino, color=blue]{VALE}\usepackage{lineno}
\definechangesauthor[name=Ferdinando, color=green]{FERD}\usepackage{lineno}
\usepackage{amssymb}

\usepackage[all,cmtip]{xy}

\usepackage{mathrsfs}

\usepackage{tabularx}
\usepackage{booktabs}
\usepackage[labelfont=bf,format=plain,justification=raggedright,singlelinecheck=false]{caption}

\theoremstyle{plain}
\newtheorem{thm}{Theorem}[section]

\newtheorem{lem}[thm]{Lemma}
\newtheorem{prop}[thm]{Proposition}

\newtheorem{cor}[thm]{Corollary}
\newtheorem{rem}[thm]{Remark}

\author{Valentino Smaldore, Corrado Zanella and Ferdinando Zullo}
\title{New scattered linearized quadrinomials}
\date{\today}

\newcommand{\F}{\mathbb F}
\newcommand{\Fq}{{{\mathbb F}_q}}
\newcommand{\fq}{{{\mathbb F}_q}}

\newcommand{\Fqn}{{{\mathbb F}_{q^n}}}
\newcommand{\fqn}{{{\mathbb F}_{q^n}}}
\newcommand{\Fqt}{{{\mathbb F}_{q^t}}}
\newcommand{\Fqts}{{{\mathbb F}_{q^t}^*}}

\newcommand{\la}{\langle}
\newcommand{\ra}{\rangle}

\newcommand{\NN}{\mathbb N}
\newcommand{\Rt}{\operatorname{R-}q^t\operatorname-}
\newcommand{\Lt}{\operatorname{L-}q^t\operatorname-}
\newcommand{\pol}{\varphi_{m,q}}
\newcommand{\poll}{\varphi_{m,\sigma}}
\newcommand{\polll}{\varphi_{m,{q^J}}}
\newcommand{\trea}{(iii\text{-}a)}
\newcommand{\treb}{(iii\text{-}b)}
\newcommand{\sd}[1]{{\sigma^{#1}}}

\DeclareMathOperator{\PG}{PG}
\DeclareMathOperator{\N}{N}
\DeclareMathOperator{\Tr}{Tr}
\DeclareMathOperator{\GaL}{\Gamma L}
\DeclareMathOperator{\GL}{GL}

\DeclareMathOperator{\C}{\mathcal{C}}
\DeclareMathOperator{\Gal}{Gal}

\begin{document}

\maketitle

\begin{abstract}
Let $1<t<n$ be integers, where $t$ is a divisor of $n$.
An $\Rt$partially scattered polynomial is a $q$-polynomial $f$ in $\F_{q^n}[X]$ that satisfies the condition that for all $x,y\in\F_{q^n}^*$ such that $x/y\in\F_{q^t}$, if $f(x)/x=f(y)/y$, then $x/y\in\Fq$; $f$ is called scattered if this implication holds for all $x,y\in\F_{q^n}^*$.
Two polynomials in $\F_{q^n}[X]$ are said to be equivalent if their graphs are in the same orbit under the action of the group $\GaL(2,q^n)$.
For $n>8$ only three families of scattered polynomials in $\F_{q^n}[X]$ are known: $(i)$~monomials of pseudoregulus type, $(ii)$~binomials of Lunardon-Polverino type, and $(iii)$~a family of quadrinomials defined in \cite{BZZ,LZ} and extended in \cite{LMTZ,NSZ}.
In this paper we prove that the polynomial $\polll=X^{q^{J(t-1)}}+X^{q^{J(2t-1)}}+m(X^{q^J}-X^{q^{J(t+1)}})\in\F_{q^{2t}}[X]$, $q$ odd, $t\ge3$ is $\Rt$partially scattered for every value of $m\in\F_{q^t}^*$ and $J$ coprime with $2t$.
Moreover, for every $t>4$ and $q>5$ there exist values of $m$ for which $\pol$ is scattered and new with respect to the polynomials mentioned in $(i)$, $(ii)$ and $(iii)$ above.
The related linear sets are of $\GaL$-class at least two.
\end{abstract}


\section{Introduction}

In this paper $q=p^{\varepsilon}$ denotes a power of a prime $p$.
Let $n$ be a positive integer.
A \emph{$q$-polynomial}, or \emph{$\fq$-linearized polynomial}, in $\fqn[X]$ is of type
\begin{equation}\label{e:qpoly}
  f=\sum_{i=0}^d a_iX^{q^i},\quad a_i\in\Fqn,\ i=0,1,\ldots,d.
\end{equation}
If $a_d\neq0$, then the $q$-polynomial $f$ has \emph{$q$-degree} $d$.
The set of $q$-polynomials over $\fqn$ is denoted by $L_{n,q}$.
Such a set, equipped with the operations of sum, multiplication by elements of $\fq$, and the composition, results to be an $\fq$-algebra.
The quotient algebra $\mathcal{L}_{n,q}=L_{n,q}/(X^{q^n}-X)$ is isomorphic to the $\fq$-algebra of the $\fq$-linear endomorphisms of $\fqn$.
Hence, for every $\fq$-linear endomorphism $\Phi$ of $\fqn$ there exists a unique $q$-polynomial $f$ of $q$-degree less than $n$ such that $f(x)=\Phi(x)$ for any $x\in\fqn$.
By abuse of notation, any $f\in L_{n,q}$ is identified with the class $f+(X^{q^n}-X)\in\mathcal L_{n,q}$.

$q$-polynomials have found many applications in various areas of combinatorics, coding theory and cryptography.
Special attention has been paid to \emph{scattered polynomials}, introduced by Sheekey in \cite{Sheekey}, in the context of optimal codes in the rank metric, as we will see later.
A \emph{scattered polynomial} is an  $f\in\mathcal{L}_{n,q}$ such that for any $x,y\in\mathbb F_{q^n}^*$,
\begin{equation}\label{eq:condscatt}
\frac{f(x)}{x}=\frac{f(y)}{y}\ \Longrightarrow\ \frac{x}{y}\in\mathbb{F}_{q}.
\end{equation}

A generalization of the notion of scattered polynomials was recently introduced in \cite{LZ2}.
Let $f$ be a $q$-polynomial in $\mathcal{L}_{n,q}$,
and $t$ a nontrivial divisor of $n$; so, $n=tt'$, and $1<t,t'<n$.
We say that $f$ is \emph{$\Lt$partially scattered} if for any $x,y\in\mathbb F_{q^n}^*$,
\begin{equation}\label{eq:condL}
\frac{f(x)}{x}=\frac{f(y)}{y}\ \Longrightarrow\ \frac{x}{y}\in\mathbb{F}_{q^t},
\end{equation}
and that $f$ is \emph{$\Rt$partially scattered} if for any  $x,y\in\mathbb F_{q^n}^*$,
\begin{equation}\label{eq:condR}
\frac{f(x)}{x}=\frac{f(y)}{y}\,\mbox{ and }\, \frac{x}{y}\in\mathbb{F}_{q^t}\ \Longrightarrow\  \frac{x}{y}\in\fq.
\end{equation}
A  $q$-polynomial is scattered if and only if it is both $\Lt$ and $\Rt$partially scattered.

The \emph{graph} of $f\in\fqn[X]$ is $U_f=\{(x,f(x))\colon x\in\fqn\}$.
If $f\in\mathcal L_{n,q}$, then $U_f$ is an $\fq$-subspace of $\F_{q^n}^2$.
We say that two polynomials $f$ and $g$ in $\mathcal{L}_{n,q}$ are \emph{$\GaL$-equivalent}, or simply \emph{equivalent}, if their graphs are in the same orbit under the action of the group $\GaL(2,q^n)$.
Up to such a notion of equivalence,
 only three families of scattered polynomials in $\mathcal{L}_{n,q}$ are known for $n>8$:
\begin{itemize}
\item[$(i)$] $X^{q^s}$, with $\gcd(s,n)=1$, known as \emph{pseudoregulus type};
\item[$(ii)$] $X^{q^{n-s}}+\delta X^{q^s}$, with $\gcd(s,n)=1$ and $\N_{q^n/q}(\delta)=\delta^{\frac{q^n-1}{q-1}}\ne 0,1$, known as \emph{Lunardon-Polverino} type \cite{LP,Sheekey};
\item[$(iii)$] $\psi_{s,h}=X^{q^s}+X^{q^{s(t-1)}}+h^{1+q^{s}}X^{q^{s(t+1)}}+h^{1-q^{s(2t-1)}}X^{q^{s(2t-1)}} \in \mathcal{L}_{2t,q}$, where $q$ is odd, $\gcd(s,2t)=1$, $h \in \F_{q^{2t}}$ and $h^{q^t+1}=-1$ \cite{BZZ,LMTZ,LZ,NSZ}.
\end{itemize}
It is convenient to split the family $(iii)$ into two separate families $\trea$ of all polynomials $\psi_{s,h}$ such that $h\in\Fqt$, and $\treb$, where $h\notin\Fqt$.

One of the reasons that scattered polynomials and their generalizations have attracted so much attention is  their connection to optimal codes in the theory of rank-metric codes.
See \cite{PolZul,Sheekey} for a general overview of this topic.
Rank-metric codes are sets of $n \times m$ matrices over a finite field $\F_q$, endowed with the rank metric. When $n=m$, the rank-metric codes can be represented in terms of $q$-polynomials, since the $\fq$-algebra $\F_q^{n\times n}$ is isomorphic to $\mathcal L_{n,q}$.
Of particular interest is the family of \emph{maximum rank distance (MRD)} codes, which are of maximum size for  given $n$  and given minimum rank distance.
The first construction of a family of MRD codes was due to Delsarte \cite{Delsarte} and independently to Gabidulin \cite{Gabidulin}. The codes of this family are now known as the \emph{Gabidulin codes}.
Sheekey in \cite{Sheekey} pointed out a way to construct special classes of MRD codes: if $f\in\mathcal L_{n,q}$ is a scattered polynomial then $\C_f=\la X, f \ra_{\fqn}$ is an MRD code of size $q^{2n}$ and minimum distance $n-1$.
In the same paper, he also proved that the equivalence of $q$-polynomials corresponds to the equivalence of rank-metric codes. Therefore, the scattered polynomials $(i)$, $(ii)$ and $(iii)$ give rise to three disjoint families of MRD codes.

The \emph{adjoint} of a $q$-polynomial $f=\sum_{i=0}^{n-1}a_iX^{\sigma^i}$ with respect to the trace form is the polynomial $f^\top$ which satisfies $\Tr_{q^n/q}(f(x)y)=\Tr_{q^n/q}(xf^\top(y))$ for every $x,y\in\Fqn$
(\footnote{$\Tr_{q^n/q}(x)=x+x^q+\cdots+x^{q^{n-1}}$ for $q\in\Fqn$.}).
It holds
\[
 f^\top=\sum_{i=0}^{n-1}a_i^{\sigma^{n-i}}X^{\sigma^{n-i}}.
\]
Such $f^\top$ is scattered if and only if $f$ is. Even, $f$ and $f^\top$ determine the same linear set $L_f=L_{f^\top}$ \cite{CsMaPo}.
However, the polynomials $f$ and $f^\top$ need not to be equivalent.
The classes $(i)$, $(ii)$, and $\treb$ described above contain their own adjoints up to equivalence.
Furthermore, every polynomial in the collection $\treb$ is equivalent to its adjoint \cite[Theorem~4.6]{LMTZ}, \cite[Proposition~4.17]{NSZ}.

The scattered polynomials can also be used to construct scattered linear sets.
We refer the reader to \cite{Lavrauw,Polverino,PolZul} for generalities on this topic.
Any linear set of rank $n$ in the projective line $\PG(1,q^n)$ can be defined, up to the action of GL$(2,q^n)$, by a $q$-polynomial $f$ as follows
\[ L_f=\{ \la (x,f(x)) \ra_{\fqn} \colon x \in \F_{q^n}^* \}, \]
and we say that it is \emph{scattered} if $|L_f|=(q^n-1)/(q-1)$.
It is a straightforward check that $f$ is scattered if and only if $L_f$ is scattered.
Scattered linear sets on the projective line have been related to various objects in Galois geometries, such as linear minimal blocking sets of largest order, translation planes and many more; see for example \cite{Lavrauw,LZ,Polverino}.

In this paper, we study the $q$-polynomials
\[ \poll=X^{{\sigma}^{t-1}}+X^{{\sigma}^{2t-1}}+m(X^{\sigma}-X^{{\sigma}^{t+1}})\in \mathcal{L}_{2t,q}, \]
where $t\ge3$, $m\in\F_{q^t}^*$, $x\mapsto x^\sigma$ is a generator of $\Gal(\fqn/\Fq)$, i.e. $\sigma={q^J}$, $J\in\{1,2,\ldots,2t-1\}$, $\gcd(J,2t)=1$, and $q$ is odd.
In Section \ref{sec:polisscatt}, we will first prove that any such $\poll$ is R-$q^t$-partially scattered, and then we will show some conditions on $m$ that ensure that $\poll$ is scattered.
In Section \ref{sec:stabil}, we find the stabilizer of the graph of $\poll$ under the action of the group $\mathrm{GL}(2,q^n)$, which turns out to be an invariant for equivalence.
In the last section we consider the question of equivalence between $\pol$ and the known families of scattered polynomials.
The adjoint $\varphi_{m,q}^\top$ of a scattered polynomial $\pol$ is not equivalent to $\varphi_{k,q}$ for any $k\in\Fqt$; hence, $\pol$ does not belong to the family $\treb$ (Proposition~\ref{p:class}).
The main result of this paper is Theorem~\ref{t:main}. It states that for $t>4$, if $t$ is even and $q>3$ or $t$ is odd and $q>5$, there exists at least one scattered polynomial of type $\pol$ that is not equivalent to any known scattered polynomial.

In \cite{CsMaPo} the notion of \emph{$\GaL$-class} of a linear set $L$ has been introduced, which is the number of nonequivalent polynomials $f$ such that $L_f=L$.
As a consequence of Proposition~\ref{p:class}, the $\GaL$-class of any scattered linear set of type $L_{\pol}$ is at least two.

\section{A family of $\Rt$partially scattered polynomials}\label{sec:polisscatt}
From now on we will assume that $q$ is an odd prime power, $t\ge3$ is an integer, and $n=2t$.
We will show a family of $\Rt$partially scattered polynomials
in $\mathcal{L}_{n,q}$.

For $m\in\fqn$ and $\sigma=q^J$, $J\in\{1,\ldots,n-1\}$, $\gcd(J,n)=1$,
consider the following $q$-polynomials in $\mathcal{L}_{n,q}$:
\[ \alpha=\alpha_\sigma=X^{{\sigma}^{t-1}}+X^{{\sigma}^{2t-1}}\,\,\,\text{and}\,\,\,\beta=\beta_{m,\sigma}=m(X^{\sigma}-X^{{\sigma}^{t+1}}). \]
Define
\[ W=\{ x \in \mathbb{F}_{q^{n}} \colon x^{q^t}+x=0 \}, \]
which is a one-dimensional $\F_{q^t}$-subspace of $\F_{q^{n}}$.
Note that $x^{-1},x^q,x^\sigma\in W$ for any $x\in W$, $x\neq0$.
Some results in the sequel, such as the next ones, are based on the fact that the kernel of a non-trivial $\Fq$-linear map of type $\Phi(x)=\sum_{i=0}^d a_ix^{\sigma^i}$ has dimension at most $d$ over $\Fq$ (see e.g. \cite[Theorem 5]{GQ2009}).

\begin{lem}\label{lem:techn}
The following hold:
\begin{enumerate}
    \item if $A \in \F_{q^t}$ and $B\in W$ then $AB \in W$;
    \item if $A , B\in W$ then $AB \in \F_{q^t}$;
    \item $\F_{q^{2t}}=\F_{q^t}\oplus W$;
    \item $\ker(\alpha)=W$;
    \item $\mathrm{Im}(\alpha)=\F_{q^t}$;
    \item $\ker(\beta)=\F_{q^t}$;
    \item $\mathrm{Im}(\beta)=W$
\end{enumerate}
\end{lem}
\begin{proof}
We only prove 4. and 5.
Raising to the $\sigma$ one obtains
\[
  \ker(\alpha)=\{x^{\sigma^t}+x=0\}.
\]
Furthermore, if $x\in W$, then $x^{\sigma^t}=(((x^{q^t})^{q^t})\cdots)^{q^t}$, an odd number of powers, hence $x^{\sigma^t}=-x$, that is $x\in\ker(\alpha)$.
This implies $W\subseteq\ker(\alpha)$, and 4. follows from $\dim_{\fq}\ker(\alpha)\le t$.

Next, noting that $\Fqt=\{y\in\fqn\colon y^{\sigma^t}-y=0\}$; if $y=\alpha(x)$, then
\[ y^{\sigma^t}-y=(x^{{\sigma}^{t-1}}+x^{{\sigma}^{2t-1}})^{\sigma^t}-
(x^{{\sigma}^{t-1}}+x^{{\sigma}^{2t-1}})=0. \]
\end{proof}

\begin{thm}\label{thm:Rscattpoly}
Let $t\ge3$ be an integer.
Assume $m\in\mathbb F_{q^t}^*$ and $\sigma=q^J$, $J\in\{1,\ldots,2t-1\}$, $\gcd(J,2t)=1$.
Then the polynomial
\begin{equation}\label{q:pol}
  \poll=X^{{\sigma}^{t-1}}+X^{{\sigma}^{2t-1}}+m(X^{\sigma}-X^{{\sigma}^{t+1}})\in\mathcal{L}_{2t,q}
\end{equation}
is $\Rt$partially scattered.
\end{thm}
\begin{proof}
The polynomial $\poll=\alpha+\beta$ is $\Rt$partially scattered if and only if $\poll$ satisfies the condition that for any $\rho \in \F_{q^t}$ and $x \in \F_{q^{n}}$ such that $x\ne 0$, if
\begin{equation}\label{eq:condscattR} \poll(\rho x)=\rho\poll(x),
\end{equation}
then $\rho \in \F_q$.
So, suppose that \eqref{eq:condscattR} holds and, because of 3.\ of Lemma \ref{lem:techn}, we can write
$x=x_1+x_2$,
where $x_1 \in \F_{q^t}$ and $x_2 \in W$.
Using 1., 2., 4.\ and 5.\ of Lemma \ref{lem:techn} we obtain
\[\poll(\rho x)=\alpha(\rho x_1)+\beta(\rho x_2)\]
and
\[ \rho\poll(x)=\rho\alpha(x_1)+\rho\beta(x_2). \]
Since $\F_{q^{2t}}=\F_{q^t}\oplus W$, by 5.\ and 7.\ of Lemma \ref{lem:techn} we have
\begin{equation}\label{eq:systFqt+WR}
\left\{
\begin{array}{ll}
\alpha(\rho x_1)=\rho\alpha(x_1),\\
\beta(\rho x_2)=\rho\beta(x_2),
\end{array}
\right.
\end{equation}
which can be rewritten as
\[
\left\{
\begin{array}{ll}
(\rho^{{\sigma}^{t-1}}-\rho)x_1^{{\sigma}^{t-1}}=0,\\
m(\rho^{\sigma}-\rho)x_2^{\sigma}=0,
\end{array}
\right.
\]
and since $m \ne 0$ and at least one among $x_1$ and $x_2$ is nonzero, then
$\rho^{{\sigma}^{t-1}}-\rho=0$ or $\rho^{\sigma}-\rho=0$. In each case we get $\rho \in \fq$.
\end{proof}

In the following we show that for certain values of $m$ the polynomial $\poll$ is also $\Lt$partially scattered.

\begin{thm}\label{thm:scattpoly}
Let $t\ge3$ be an integer, and $W=\{x\colon x \in \mathbb{F}_{q^{2t}} ,\, x^{q^t}+x=0 \}$.
Assume $\sigma=q^J$, $J\in\{1,\ldots,2t-1\}$, $\gcd(J,2t)=1$.
If $m\in\Fqt$ is neither a $(q - 1)$-th power nor a
$(q + 1)$-th power of any element of $W$
then the polynomial $\poll=X^{{\sigma}^{t-1}}+X^{{\sigma}^{2t-1}}+m(X^{\sigma}-X^{{\sigma}^{t+1}})\in\mathcal{L}_{2t,q}$ is scattered.
\end{thm}
\begin{proof}
By Theorem~\ref{thm:Rscattpoly}, it is enough to prove that $\poll$ is $\Lt$partially scattered;
that is, it satisfies the condition that for any
$\rho,x \in \F_{q^{n}}$ such that $x\ne 0$, if
\begin{equation}\label{eq:condscatt2} \poll(\rho x)=\rho\poll(x),
\end{equation}
then $\rho \in \Fqt$.
So, suppose that \eqref{eq:condscatt2} holds. Because of 3.\ of Lemma \ref{lem:techn},
we can write
\[ \rho = h + r\,\,\,\text{and}\,\,\, x=x_1+x_2, \]
where $h,x_1 \in \F_{q^t}$ and $r,x_2 \in W$.
Using 1., 2., 4.\ and 5.\ of Lemma \ref{lem:techn} we obtain
\[\poll(\rho x)=\alpha(hx_1)+\alpha(rx_2)+\beta(hx_2)+\beta(rx_1)\]
and
\[ \rho\poll(x)=h\alpha(x_1)+h\beta(x_2)+r\alpha(x_1)+r\beta(x_2). \]
Since $\F_{q^{2t}}=\F_{q^t}\oplus W$, by 5. and 7. of Lemma \ref{lem:techn} we have
\begin{equation}\label{eq:systFqt+W}
\left\{
\begin{array}{ll}
\alpha(hx_1)+\alpha(rx_2)=h\alpha(x_1)+r\beta(x_2),\\
\beta(hx_2)+\beta(rx_1)=h\beta(x_2)+r\alpha(x_1),
\end{array}
\right.
\end{equation}
which can be rewritten as
\[
\left\{
\begin{array}{ll}
mx_2^{\sigma} r-x_2^{{\sigma}^{t-1}}r^{{\sigma}^{t-1}}=(h^{{\sigma}^{t-1}}-h)x_1^{{\sigma}^{t-1}},\\
x_1^{{\sigma}^{t-1}}r-mx_1^{\sigma}r^{\sigma}=-m(h-h^{\sigma})x_2^{\sigma},
\end{array}
\right.
\]
and raising the first equation to the ${\sigma}$ we obtain
\begin{equation}\label{eq:systrrq}
\left\{
\begin{array}{ll}
-x_2r+m^{\sigma}x_2^{{\sigma}^2} r^{\sigma}=(h-h^{\sigma})x_1,\\
-x_1^{{\sigma}^{t-1}}r+mx_1^{\sigma}r^{\sigma}=m(h-h^{\sigma})x_2^{\sigma}.
\end{array}
\right.
\end{equation}
Suppose by contradiction that $r\ne 0$. We divide the proof in four cases.
Since $J$ is odd, every $(\sigma-1)$-th power (resp.\ $(\sigma+1)$-th power) of an element of $W$ is also a $(q-1)$-th power (resp.\ $(q+1)$-th power) of an element of $W$.\\
\textbf{Case 1}: $x_1=0$.\\
In this case $x_2\ne 0$, and from the first equation of \eqref{eq:systrrq} we obtain
\[ m^\sigma=x_2^{1-\sigma^2}r^{1-\sigma}=\left(x_2^{-1-\sigma}r^{-1}\right)^{\sigma-1}, \]
that is $m=\delta^{q-1}$ for some $\delta \in W$, a contradiction to our assumptions.\\
\textbf{Case 2}: $x_2=0$.\\
We have $x_1\ne 0$.
From the second equation of \eqref{eq:systrrq} we obtain
\[ m=x_1^{\sigma^{t-1}-\sigma}r^{1-\sigma}=\left(x_1^{\sigma(1+\sigma+\cdots+\sigma^{t-3})}/r\right)^{\sigma-1}, \]
again a contradiction.\\
\textbf{Case 3}: $h-h^\sigma=0$. \\
Argue as in Case 1.\ or 2, depending on whether $x_2\neq0$ or $x_1\neq0$.\\
\textbf{Case 4}: $x_1x_2(h-h^{\sigma}) \ne 0$.\\
We start by proving that
\[ D=\det \left(
\begin{array}{cc}
-x_2 & m^{\sigma}x_2^{{\sigma}^2}\\
-x_1^{{\sigma}^{t-1}} & m x_1^{\sigma}
\end{array}
\right)
\]
is non zero. Indeed, if $D=0$ then, since \eqref{eq:systrrq} admits solutions
for $r$ and $r^{\sigma}$, we have
\[ \mathrm{rk}
\left(
\begin{array}{ccc}
-x_2 & m^{\sigma}x_2^{{\sigma}^2} & (h-h^{\sigma})x_1 \\
-x_1^{{\sigma}^{t-1}} & m x_1^{\sigma} & m(h-h^{\sigma})x_2^{\sigma}
\end{array}
\right)
=1,
\]
and in particular
\[ \det \left(
\begin{array}{cc}
m^{\sigma}x_2^{{\sigma}^2} & (h-h^{\sigma})x_1 \\
m x_1^{\sigma} & m(h-h^{\sigma})x_2^{\sigma}
\end{array}
\right)=0,
\]
that is
\[ m^{{\sigma}}x_2^{{\sigma}^2+{\sigma}}-x_1^{{\sigma}+1}=0\ \Rightarrow\ m=u^{\sigma+1}, \]
where
\[
u=\frac{x_1^{\sigma^{t-1}}}{x_2}\in W,
\]
a contradiction. Therefore $D \ne 0$.
From \eqref{eq:systrrq} we obtain
\[ r= \frac{\det\left(\begin{array}{cc}
(h-h^{\sigma})x_1 & m^{\sigma} x_2^{{\sigma}^2}\\
m(h-h^{\sigma})x_2^{\sigma} & mx_1^{\sigma}
\end{array}\right)}{D}=\frac{mx_1^{{\sigma}+1}-m^{{\sigma}+1}x_2^{{\sigma}^2+{\sigma}}}{D}(h-h^{\sigma}) \]
and
\[ r^{\sigma}= \frac{\det\left(\begin{array}{cc}
-x_2 & (h-h^{\sigma})x_1\\
-x_1^{{\sigma}^{t-1}} & m(h-h^{\sigma})x_2^{\sigma}
\end{array}\right)}{D}=\frac{-mx_2^{{\sigma}+1}+x_1^{{\sigma}^{t-1}+1}}{D}(h-h^{\sigma}). \]
Therefore,
\[ r^{{\sigma}-1}=\frac{-mx_2^{{\sigma}+1}+x_1^{{\sigma}^{t-1}+1}}{mx_1^{{\sigma}+1}-m^{{\sigma}+1}x_2^{{\sigma}^2+{\sigma}}}=\frac{1}m (-mx_2^{{\sigma}+1}+x_1^{{\sigma}^{t-1}+1})^{1-{\sigma}}, \]
that is $m$ is a $(\sigma-1)$-th power of an element in $W$, again a contradiction.

In each of the cases analyzed, the condition $r\neq0$ leads to a contradiction.
It follows that $\rho \in \F_{q^t}$.
\end{proof}

By the following result at least one of the assumptions above cannot be removed.
\begin{prop}\label{p:vversa}
  Let $t\geq3$. If $m$ is a $(\sigma+1)$-th power of an element of $W$, then $\poll$ is not scattered.
\end{prop}
\begin{proof}
   By assumption $m=w^{\sigma+1}$ where $w\in W$.
   Define $x_1=1$, $x_2=w^{-1}$.
   Under these assumptions the equations in \eqref{eq:systrrq} coincide up to a factor with
   \[
     -r+w^{\sigma+1}r^{\sigma}=w(h-h^\sigma).
   \]
   The images of the $\fq$-linear maps $r\in W\mapsto -r+w^{\sigma+1}r^{\sigma}\in W$ and $h\in\Fqt\mapsto w(h-h^\sigma)\in W$ both are of $\Fq$-dimension at least $t-1$; this implies that their intersection is not trivial, and $r\in W$, $h\in\Fqt$ exist such that $r\neq0$ and \eqref{eq:systrrq} is satisfied.
\end{proof}

\begin{prop}\label{p:powers}
    \begin{enumerate}[$(i)$]
    \item For any $t\ge3$ there are precisely $(q^t-1)/(q-1)$ elements of $\Fqts$ which are $(q-1)$-th powers of elements in $W$; more precisely, they are the solutions of
    \[ x^{\frac{q^t-1}{q-1}}=-1,\quad x\in\fqn.\]
    \item If $t$ is even, then there are precisely $(q^t-1)/(q+1)$ elements of $\Fqts$ which are $(q+1)$-th powers of elements in $W$; more precisely, they are the solutions of
    \[ x^{\frac{q^t-1}{q+1}}=-1,\quad x\in\fqn.\]
    \item If $t$ is odd, then there are precisely $(q^t-1)/2$ elements of $\Fqts$ which are $(q+1)$-th powers of elements in $W$; more precisely, they are the solutions of
    \[ x^{\frac{q^t-1}{2}}=(-1)^{\frac{q+1}2},\quad x\in\fqn.\]
    \end{enumerate}
\end{prop}
\begin{proof}
    Let $W^*=W\setminus\{0\}$.
    For any positive integer $D $ define the set $S_D $ of all $D $-powers of elements of $W^*$.
    Let $\delta=\gcd(D ,q^t-1)$.
    The $D $-powers of elements of $\Fqts$ are precisely the solutions of the equation
    \[ x^{\frac{q^t-1}{\delta}}=1,\quad x\in\Fqn.\]
    Let $w_0\in W^*$.
    It holds $(w_0^D )^{\frac{q^t-1}{\delta}}=(w_0^{q^t-1})^{D /\delta}=(-1)^{D /\delta}$.
    We have
    \[
      S_D =\{w_0^D  y^D \colon y\in\Fqts\}=\{w_0^D  x\colon x^{\frac{q^t-1}{\delta}}=1,\,x\in\Fqn\}.
    \]
    Therefore, $S_D $ has equation
    \begin{equation}\label{e:pow}
        x^{\frac{q^t-1}{\delta}}=(-1)^{D /\delta}.
    \end{equation}
      Taking into account that
      \[
        \gcd(q+1,q^t-1)=\left\{ \begin{array}{cc} q+1&\mbox{ for $t$ even,}\\ 2&\mbox{ for $t$ odd,}\end{array}\right.
      \]
      the statements $(i)$, $(ii)$, and $(iii)$ follow from \eqref{e:pow}.
\end{proof}
\begin{cor}\label{c:two}
    There is at least one scattered polynomial of type $\poll$
    for any $t\ge4$ even and $q\ge3$, or $t\ge3$ and $q>3$.
\end{cor}
\begin{proof}
    The sum of the sizes of $S_{q-1}$ and $S_{q+1}$ is less than $q^t-1$, hence there exist in $\Fqts$ elements which are neither $(q-1)$- nor $(q+1)$-powers of elements of $W$.
    Therefore, Theorem~\ref{thm:scattpoly} can be applied for at least one value of $m$.
\end{proof}

\begin{rem}
By Theorem~\ref{thm:scattpoly} and Proposition~\ref{p:powers}, if $t$ is even, or $t$ is odd and $q \equiv 1 \pmod{4}$, then $\varphi_{1,q}$ is scattered.
On the other hand, by Proposition~\ref{p:vversa}, if $t$ is odd and $q \equiv 3 \pmod{4}$, then $\varphi_{1,q}$ is not scattered.
This is consistent with \cite[Theorem 2.4]{LZ} for $k=1$.
It follows that that the family we are studying contains examples of R-$q^t$-partially scattered polynomials that are not scattered.
\end{rem}

\section{Matrices stabilizing the graph of $\poll$}\label{sec:stabil}

Here we investigate the matrices in $\F_{q^n}^{2\times2}$ that stabilize the graph $U_{m,\sigma}$ of $\poll$.
More precisely, we will compute the set $\mathbb S_{m,\sigma}$ of matrices $A$ such that $AU_{m,\sigma}\subseteq U_{m,\sigma}$.
Such $\mathbb S_{m,\sigma}$ has been studied in \cite{SZZ}, where in particular it has been proved that it is isomorphic to the right idealizer of $\mathcal C_{\poll}$.
As a consequence, $\mathbb S_{m,\sigma}$ is invariant up to equivalence of polynomials.

For use in the Theorem~\ref{t:stab} we prove the following
\begin{prop}\label{p:premstab}
Let $t$ be even.
The set $S$ of all $x\in\Fqts$ which are $(\sigma+1)$-powers of elements in $W$ coincides with
$
  \left\{x\colon x\in\Fqt,\,x^{\frac{\sigma^t-1}{\sigma+1}}=-1\right\}.
$
\end{prop}
\begin{proof}
    If $x_0=\xi^{\sigma+1}$ and $y=\eta^{\sigma+1}$ for $\xi,\eta\in W\setminus\{0\}$, then $x_0y^{-1}=(\xi\eta^{-1})^{\sigma+1}$ is a $(\sigma+1)$-power of an element of $\Fqts$; that is, $y=\ell^{\sigma+1}x_0$ for some $\ell\in\Fqts$.
    Conversely, $h^{\sigma+1}x_0\in S$ for any $h\in\Fqts$.
    So,
    \[
    S=\{h^{\sigma+1}x_0\colon h\in\Fqts\}.
    \]
    The assertion follows by combining $(i)$~$x_0^{\frac{\sigma^t-1}{\sigma+1}}=\xi^{\sigma^t-1}=-1$, and $(ii)$~since $t$ is even and $\sigma+1$ divides $\sigma^t-1$, the set of all $(\sigma+1)$-powers of elements of $\Fqts$ has equation $x^{\frac{\sigma^t-1}{\sigma+1}}=1$.
\end{proof}

\begin{thm}\label{t:stab}
  Suppose that $t>4$.
  Then the set $\mathbb S_{m,\sigma}$ of matrices  $A\in\F_{q^n}^{2\times2}$ such that $AU_{m,\sigma}\subseteq U_{m,\sigma}$ is
  equal to
\begin{equation}\label{e:stab}
\left\{\begin{pmatrix}a&b\\ 4mb^\sigma&a^\sigma\end{pmatrix}\ \colon\ a\in\F_{q^{\gcd(t,2)}},\ b\in\F_{q^n},\ b=-b^{\sigma^t}=m^{\sigma-1}b^{\sigma^2}\right\}.
\end{equation}
For $t$ even, $\mathbb S_{m,\sigma}$ contains non-diagonal matrices if and only if $m$ is a $(\sigma+1)$-power of an element of $W$; in this case, $b$ takes $q^2$ distinct values.

For $t$ odd, $b$ takes always $q$ distinct values.
\end{thm}

\begin{proof}
  Let $A=\begin{pmatrix}a&b\\ c&d\end{pmatrix}\in\F_{q^n}^{2\times2}$, and
  \[ A\begin{pmatrix}x\\ \poll(x)\end{pmatrix}=\begin{pmatrix}y\\ \poll(y)\end{pmatrix}; \]
  that is,
  $cx+d\poll(x)=\poll(ax+b\poll(x))$ for all $x\in\fqn$.
  This leads, after reducing modulo $X^{\sigma^{2t}}-X$, to  the following polynomial identity
\begin{align}
cX+&d(X^{{\sigma}^{t-1}}+X^{{\sigma}^{2t-1}}+m(X^{\sigma}-X^{{\sigma}^{t+1}}))=\notag\\
&a^{{\sigma}^{t-1}}X^{{\sigma}^{t-1}}+b^{{\sigma}^{t-1}}
(X^{{\sigma}^{2t-2}}+X^{{\sigma}^{t-2}}+m^{{\sigma}^{t-1}}(X^{{\sigma}^{t}}-X))\notag\\
&+a^{{\sigma}^{2t-1}}X^{{\sigma}^{2t-1}}+b^{{\sigma}^{2t-1}}
(X^{{\sigma}^{t-2}}+X^{{\sigma}^{2t-2}}+m^{{\sigma}^{2t-1}}(X-X^{{\sigma}^{t}}))\notag\\
&+ma^{{\sigma}}X^{{\sigma}}+mb^{{\sigma}}
(X^{{\sigma}^{t}}+X+m^{{\sigma}}(X^{{\sigma}^{2}}-X^{{\sigma}^{t+2}}))\notag\\
&-ma^{{\sigma}^{t+1}}X^{{\sigma}^{t+1}}-mb^{{\sigma}^{t+1}}
(X+X^{{\sigma}^{t}}+m^{{\sigma}^{t+1}}(X^{{\sigma}^{t+2}}-X^{{\sigma}^2})).\label{e:207a}
\end{align}
Taking into account the coefficients of monomials of the same degree one obtains
the  ten equations
\begin{align}
    c&=-m^\sd{t-1}b^\sd{t-1}+m^\sd{2t-1}b^\sd{2t-1}+mb^\sigma-mb^\sd{t+1}\label{sd0}\tag{e:0}\\
    md&=ma^\sigma\label{sd1}\tag{e:1}\\
    0&=m^{\sigma+1}b^\sigma+m^{1+\sd{t+1}}b^\sd{t+1}\label{sd2}\tag{e:2}\\
    0&=b^\sd{t-1}+b^\sd{2t-1}\label{sdt-2}\tag{e:t-2}\\
    d&=a^\sd{t-1}\label{sdt-1}\tag{e:t-1}\\
    0&=m^\sd{t-1}b^\sd{t-1}-m^\sd{2t-1}b^\sd{2t-1}+mb^\sigma-mb^\sd{t+1}\label{sdt}\tag{e:t}\\
    -md&=-ma^\sd{t+1}\label{sdt+1}\tag{e:t+1}\\
    0&=-m^{\sigma+1}b^\sigma-m^{1+\sd{t+1}}b^\sd{t+1}\label{sdt+2}\tag{e:t+2}\\
    0&=b^\sd{t-1}+b^\sd{2t-1}\label{sd2t-2}\tag{e:2t-2}\\
    d&=a^\sd{2t-1}\label{sd2t-1}\tag{e:2t-1}
\end{align}
The equations \eqref{sd1}, \eqref{sdt-1}, \eqref{sdt+1},  \eqref{sd2t-1}
are equivalent to
$a\in\F_{q^{\gcd(t,2)}}$, $d=a^q$.
The equations \eqref{sd2}, \eqref{sdt-2}, \eqref{sdt+2}, and \eqref{sd2t-2}
are equivalent to $b\in W$.
Then \eqref{sdt} is equivalent to $m^\sd{t-1}b^\sd{t-1}+mb^\sigma=0$, or
\begin{equation}\label{e:bs2}
    b^\sd 2=m^{1-\sigma}b.
\end{equation}
Equations \eqref{sdt} and \eqref{sd0} imply $c=2(mb^\sigma-m b^\sd{t+1})=4mb^\sigma$, leading to \eqref{e:stab}.

The equation \eqref{e:bs2} in the unknown $b$ determines the kernel of an $\F_{q^2}$-linear map, and has one or $q^2$ solutions.
Since $W$ is an $\Fqt$-linear subspace, the number of allowable values $b$ in \eqref{e:stab} is either one or $q^{\gcd(t,2)}$.

Assume \textbf{$t$ even}.
It holds
\[
-b=b^\sd t=m^{-\sd{t-1}+\sd{t-2}-\cdots-\sigma+1}b.
\]
Assume that \eqref{e:bs2} has at least one nonzero solution (and hence $q^2$ solutions).
Then \[m^{(\sigma-1)(1+\sigma^2+\cdots+\sd{t-2})}=-1,\] equivalent to
$m^{\frac{\sigma^t-1}{\sigma+1}}=-1$, that is, by Proposition~\ref{p:premstab},
$m$ is a $(\sigma+1)$-power of an element of $W$.
Conversely if
$m=\beta^{-(\sigma+1)}$ for $\beta\in W$, then \eqref{e:bs2} has the nonzero solution $b=\beta$, hence $q^2$ solutions in $W$.

Assume \textbf{$t$ odd}.
\begin{comment}
We have to prove that the equations
\[
-b^\sigma=b^{\sigma^{t+1}}=m^{(1-\sigma)(1+\sigma^2+\cdots+\sd{t-1})}b=m^{-\frac{\sd{t+1}-1}{\sigma+1}}b.
\]
have $q$ solutions.
For $b\neq0$,
the equation is equivalent to
\begin{equation}\label{e:cisiamo}
m^{\frac{\sd{t+1}-1}{\sigma+1}}=-b^{1-\sigma}
\end{equation}
\end{comment}
Define $R=-\frac{\sigma^{t+1}-1}{\sigma^2-1}$ which is an integer.
Furthermore, a $z\in\F_{q^n}^*$ exists satisfying $z^\sigma+z=0$, and $z\in W$.
Then by a direct check the solutions in $W$ to \eqref{e:bs2} are $b=\lambda zm^R$ for $\lambda\in\Fq$.
\end{proof}

\begin{rem}
  If $t$ is even (including now the case $t=4)$ and $\poll$ is scattered, then $\poll$ is in standard form with respect to the subfield $\mathbb{F}_{q^2}$; that is, $L=2$ is the greatest integer such that $\poll=F(X^{q^s})$ where $F$ is a $q^L$-polynomial, and $\gcd(s,L)=1$. This implies that the set
  of matrices stabilizing $U_{m,\sigma}$ is isomorphic to $\F_{q^2}$ \cite{LZ3,SZZ}.

  In particular: $(i)$~if $m=1$ and $t$ is even, the matrices are all diagonal; $(ii)$~if $m=1$ and $t$ is odd, then the conditions on $b$ are equivalent to $b^q+b=0$.
\end{rem}

\begin{rem}
For the case $t\ge3$ is odd, $m=1$ and $q\equiv3\pmod4$,
it can be shown that in this case the kernel and the image of any matrix of rank one in $\mathbb S_{1,\sigma}$ are points of $\PG(1,q^n)$ of weight $n/2$, i.e., such that their intersection with the graph is an $\Fq$-subspace of dimension $n/2$.
Polynomials of this type have been studied in \cite{SZZ}.
\end{rem}

\section{Nonequivalence with previously known scattered polynomials}

Our main goal is to show that there are new scattered polynomials in the family we introduce.
For this purpose it will be sufficient to consider the case where $\sigma=q$, $n=2t$, $t\ge3$.
We compare our construction with the known examples of scattered polynomials.

The first nonequivalence is a simple consequence of the fact that the right idealizer of the MRD code associated with a polynomial of pseudoregulus type in $\mathcal L_{n,q}$ is isomorphic to $\Fqn$, combined with Theorem~\ref{t:stab}. Analogously we proceed with the second nonequivalence if $t$ is odd, since the right idealizer of the MRD code associated with a polynomial of Lunardon-Polverino type in $\mathcal L_{n,q}$ is isomorphic to $\mathbb{F}_{q^2}$.
\begin{prop}\label{propeq:pseudoregulus}
Let $f=X^{q^s}\in\mathcal L_{n,q}$ be a scattered polynomial of {pseudoregulus type}. Then $\pol$ and ${f}$  are not equivalent.
\end{prop}

\begin{prop}\label{propeq:LP}
Let $g=X^{q^{2t-s}}+\delta X^{q^s}$ be a \textit{Lunardon-Polverino} scattered polynomial, $t>4$. Then ${\pol}$ and ${g}$  are not equivalent.
\end{prop}
\begin{proof}
The subspaces $U_{\pol}$ and $U_g$ are in the same orbit under the action of $\mathrm{\Gamma L}(2,q^n)$ if and only if an integer $k$ and a matrix\[
  M=\begin{pmatrix}a&b\\ c&d\end{pmatrix}\in\GL(2,q^{2t})
\]
exist such that for any $x\in\F_{q^{2t}}$ there is $y\in\F_{q^{2t}}$  satisfying

\[ M
\left(
\begin{array}{c}
    x^{p^k} \\
    \pol(x)^{p^k}
\end{array}
\right)=
\left(
\begin{array}{c}
    y \\
    g(y)
\end{array}
\right).
\]
Let $e=m^{p^k}$ and $z=x^{p^k}$. The condition above is equivalent to
\[
az+b(z^{q^{t-1}}+z^{q^{2t-1}})+be(z^q-z^{q^{t+1}})=y,
\]
\[
cz+d(z^{q^{t-1}}+z^{q^{2t-1}})+de(z^q-z^{q^{t+1}})=y^{q^{2t-s}}+\delta y^{q^s},
\]
from which after reducing modulo $Z^{\sigma^{2t}}-Z$ we derive the polynomial identity
\[ a^{q^{2t-s}}Z^{q^{2t-s}}+b^{q^{2t-s}}(Z^{q^{t-s-1}}+Z^{q^{2t-s-1}})+b^{q^{2t-s}}e^{q^{2t-s}}(Z^{q^{2t-s+1}}-Z^{q^{t-s+1}})+ \]
\[ +\delta a^{q^s}Z^{q^s}+\delta b^{q^s}(Z^{q^{t+s-1}}+Z^{q^{s-1}})+\delta b^{q^s}e^{q^s}(Z^{q^{s+1}}-Z^{q^{t+s+1}})= \]
\[ =cZ+d(Z^{q^{t-1}}+Z^{q^{2t-1}})+de(Z^q-Z^{q^{t+1}}). \]
Let $t$ be even. We consider the cases $s\in\{1,t-1,t+1,2t-1\}$.\\
Let $s=1$. When $t>4$, we have the following system
\begin{equation}\label{eq:LP_s=1}
\left\{
\begin{array}{lllllll}
b^{q^{2t-1}}e^{q^{2t-1}}+\delta b^q=c\\
\delta a^q=de\\
\delta b^q e^q=0\\
b^{q^{2t-1}}=0\\
0=d\\
-b^{q^{2t-1}}e^{q^{2t-1}}+\delta b^q=0\\
0=-de\\
-\delta b^q e^q=0\\
b^{q^{2t-1}}=0\\
a^{q^{2t-1}}=d.
\end{array}
\right.
\end{equation}
Which implies $a=b=c=d=0$.
For $s\in\{t-1,t+1,2t-1\}$ we get analogous conditions that yield to $a=b=c=d=0$.

Finally, for $s\notin\{1,t-1,t+1,2t-1\}$ the analogous system as in \eqref{eq:LP_s=1} leads to the same conclusion. In fact, the exponents of the indeterminate $Z$ depending on $s$ equal some of the exponents of $Z$ in the right side of the polynomial identity $(1,q,q^{t+1},q^{t-1},q^{2t-1})$ if and only if $s\in\{1,t+1,t-1,2t-1,t+2,t-2\}$. Since $t$ is even we exclude $t-2$ and $t+2$, and then,
apart from the already considered cases, we get the condition $c=d=0$. The case $t$ odd arises from similar calculations while $s\in\{1,t-2,t+2\}$, while in this case we can exclude $t=\pm1$ and $t=2t-1$.
\end{proof}

The fact that $\varphi_{1,\sigma}$ belongs to the family $\trea$ motivates our interest in the next result.

\begin{prop}\label{prop:equivh=1}
Let $t\in\NN$, $t>4$.
Then  ${\pol}$ and ${\varphi_{1,\sigma}}$ are equivalent only if $\N_{q^t/q}(m)=1$.
\end{prop}

\begin{proof}
Assume that ${\pol}$ and ${\varphi_{1,\sigma}}$ are equivalent.
Then there exist an integer $k$ and a matrix
\[
  M=\begin{pmatrix}a&b\\ c&d\end{pmatrix}\in\GL(2,q^{n})
\]
such that for any $x\in\F_{q^{2t}}$ there is $y\in\F_{q^{2t}}$  satisfying
\[ M
\left(
\begin{array}{c}
    x^{p^k} \\
    \pol(x)^{p^k}
\end{array}
\right)=
\left(
\begin{array}{c}
    y \\
    \varphi_{1,\sigma}(y)
\end{array}
\right).
\]

Let $e=m^{p^k}$ and $z=x^{p^k}$.
The condition above is equivalent to
\[
az+b(z^{q^{t-1}}+z^{q^{2t-1}})+be(z^q-z^{q^{t+1}})=y,
\]
\[
cz+d(z^{q^{t-1}}+z^{q^{2t-1}})+de(z^q-z^{q^{t+1}})=y^{\sigma}-y^{{\sigma}^{t+1}}+y^{{\sigma}^{t-1}}+y^{{\sigma}^{2t-1}},
\]
and, taking into account $e^{{\sigma}^t}=e$ and that $J$ is odd, we get the following identity in $\mathcal L_{n,q}$:
\[cZ+d(Z^{q^{t-1}}+Z^{q^{2t-1}})+de(Z^q-Z^{q^{t+1}})=\]
\[=a^{\sigma}Z^{q^J}+b^{\sigma}(Z^{q^{t+J-1}}+Z^{q^{J-1}})+b^{\sigma}e^{\sigma}(Z^{q^{J+1}}-Z^{q^{t+J+1}})\]
\[-a^{{\sigma}^{t+1}}Z^{q^{t+J}}-b^{{\sigma}^{t+1}}(Z^{q^{J-1}}+Z^{q^{t+J-1}})-b^{{\sigma}^{t+1}}e^{\sigma}(Z^{q^{t+J+1}}-Z^{q^{J+1}})\]
\[+a^{{\sigma}^{t-1}}Z^{q^{t-J}}+b^{{\sigma}^{t-1}}(Z^{q^{2t-J-1}}+Z^{q^{t-J-1}})+b^{{\sigma}^{t-1}}e^{{\sigma}^{t-1}}(Z^{q^{t-J+1}}-Z^{q^{2t-J+1}})\]
\begin{equation}+a^{{\sigma}^{2t-1}}Z^{q^{2t-J}}+b^{{\sigma}^{2t-1}}(Z^{q^{t-J-1}}+Z^{q^{2t-J-1}})+b^{{\sigma}^{2t-1}}e^{{\sigma}^{2t-1}}(Z^{q^{2t-J+1}}-Z^{q^{t-J+1}}).\label{e:polid}
\end{equation}
By comparing the monomials having the same degree, if $J\notin\{\pm1,t\pm1\}$, one obtains without any assumption on $\N_{q^t/q}(m)$ that $M$ has a zero row.

If $J=1$, by \eqref{e:polid},
\begin{equation}\label{eq:4.3_t>4}
\left\{
\begin{array}{lllllll}
 c=b^q-b^{q^{t+1}}-b^{q^{t-1}}e^{q^{t-1}}+b^{q^{2t-1}}e^{q^{t-1}}\\
 de=a^q\\
 0=b^q+b^{q^{t+1}}\\
 0=b^{q^{t-1}}+b^{q^{2t-1}}\\
 d=a^{q^{t-1}}\\
 0=b^q-b^{q^{t+1}}+b^{q^{t-1}}e^{q^{t-1}}-b^{q^{2t-1}}e^{q^{t-1}}\\
 de=a^{q^{t+1}}\\
 0=-b^qe^q-b^{q^{t+1}}e^q\\
 0=b^{q^{t-1}}+b^{q^{2t-1}}\\
 d=a^{q^{2t-1}}.
\end{array}
\right.
\end{equation}

If $a\neq0$, then $d\neq0$. By comparing the seventh and the second equation, one obtains $a\in\Fqt$.
From the second and the fifth equation, $a^{q-q^{t-1}}=e$ that implies $\N_{q^t/q}(m)=1$.
If $a=0$, then $cb\neq0$, $b^{q^t}+b=0$ by the third equation and $b^{q^2}=eb$ by the sixth.
This implies $\N_{q^t/q}(m)=1$.

The cases $J=2t-1$ or $J=t\pm1$ for $t$ even can be dealt with in a similar way leading in each case to $\N_{q^t/q}(m)=1$.
\end{proof}

We now investigate the equivalence between polynomials of type $\pol$ and their adjoints.
Define $\phi_{\mu}=X^q+X^{q^{t+1}}+\mu(X^{q^{2t-1}}-X^{q^{t-1}})$, $\mu\in\Fqt$; it holds
\begin{equation}
  \varphi_{m,q}^\top=\phi_{m^{q^{t-1}}}
\end{equation}
for any $m\in\Fqt$.
\begin{prop}\label{p:class}
Let $t>4$.
  For any $m,\mu\in\Fqt$ such that $\pol$ is scattered, $\pol$ and $\phi_\mu$ are nonequivalent.
\end{prop}
\begin{proof}
  Starting from
  \[
  \begin{pmatrix}a&b\\ c&d\end{pmatrix}
  \begin{pmatrix}x\\ x^{q^{t-1}}+x^{q^{2t-1}}+m(x^q-x^{q^{t+1}})\end{pmatrix}=
  \begin{pmatrix}y\\ y^{q}+y^{q^{t+1}}+\mu(y^{q^{2t-1}}-y^{q^{t-1}})\end{pmatrix}
  \]
  one obtains the polynomial identity modulo $x^{\sigma^{2t}}-x$
\begin{align*}
  cX+&d[X^{q^{t-1}}+X^{q^{2t-1}}+m(X^q-X^{q^{t+1}})]=
  a^qX^q+b^q[X^{q^t}+X+m^q(X^{q^2}-X^{q^{t+2}})]\\
  &+a^{q^{t+1}}X^{q^{t+1}}+b^{q^{t+1}}[X+X^{q^t}+m^{q^{t+1}}(X^{q^{t+2}}-X^{q^2})]\\
  &+\mu a^{q^{2t-1}}X^{q^{2t-1}}+\mu b^{q^{2t-1}}[X^{q^{t-2}}+X^{q^{2t-2}}+m^{2^{2t-1}}(X-X^{q^t})]\\
&-\mu a^{q^{t-1}}X^{q^{t-1}}-\mu b^{q^{t-1}}[X^{q^{2t-2}}+X^{q^{t-2}}+m^{q^{t-1}}(X^{q^t}-X)].
\end{align*}
This gives ten equations, four equations equivalent to $b\in\Fqt$ and four equations equivalent to $d=m^{-1}a^q=-\mu a^{q^{t-1}}=-m^{-1}a^{q^{t+1}}$, that is
  \[
    a\in W,\ d=m^{-1}a^q,\ \mu=-m^{-1}a^{q-q^{t-1}}.
  \]
 By comparing the coefficients of $X^{q^t}$ one obtains $\mu=m^{-q^{t-1}}b^{q-q^{t-1}}$, and finally
 \[
   m^{q-1}=\left(\frac ab\right)^{q^2-1}.
 \]
 As a consequence, $m^{(q^t-1)/(q+1)}=(a/b)^{q^t-1}=-1$ for $t$ even end $m^{(q^t-1)/2}=(a/b)^{(q^t-1)(q+1)/2}=(-1)^{(q+1)/2}$ for $t$ odd and this implies by Propositions~\ref{p:vversa} and \ref{p:powers} that $\pol$ is not scattered.
\end{proof}
  In \cite{CsMaPo} the notion of a \emph{$\GaL$-class} of a linear set $L$ has been introduced, which is the number of nonequivalent polynomials $f$ such that $L_f=L$.
  As a corollary of  Proposition~\ref{p:class} it results:
\begin{thm}
For $t>4$ the $\GaL$-class of any scattered linear set of type $L_{\pol}$ is at least two.
\end{thm}

It is possible to prove the nonequivalence of the polynomials of type $\pol$ with the polynomials in the class $\treb$ with a case-by-case analysis. However, the following result makes it possible to shorten the proof.
\begin{prop}\cite{unpub}\label{prop:unpub}
Let $f,g\in\mathcal L_{n,q}$ be equivalent. Then $f^\top$ and $g^\top$ are equivalent.
\end{prop}

The main result of this paper is a summary of the propositions of this section.

\begin{thm}\label{t:main}
  Let $q$ be odd and $t>4$. If $t$ is even and $q>3$ or $t$ is odd and $q>5$, then there exists $m\in\Fqt$ such that \[\pol =X^{{q}^{t-1}}+X^{{q}^{2t-1}}+m(X^{q}-X^{{q}^{t+1}}) \]
  is a scattered $q$-polynomial that is not equivalent to any previously known scattered $q$-polynomial in $\F_{q^{2t}}[X]$.
\end{thm}
\begin{proof}
  Take into account a scattered $\pol$.
  This $\pol$ does not belong to the families $(i)$ and $(ii)$  by Propositions~\ref{propeq:pseudoregulus} and \ref{propeq:LP}.
  Since any element of the family $\treb$ is equivalent to its adjoint, by Propositions~\ref{p:class} and \ref{prop:unpub} such family does not contain $\pol$.
    The family $\trea$ contains elements of type $\varphi_{1,\sigma}$ and if $\N_{q^t/q^{\gcd(2,t)}}(m)\neq1$,
   $\pol$ is nonequivalent to these by Proposition~\ref{prop:equivh=1}.
So it remains to prove that at least one scattered $\pol$ exists satisfying that condition.
    Taking into account Theorem~\ref{thm:scattpoly},
    it is enough to show that the sum of the cardinalities of the following three sets is less than $q^t-1$.
    \begin{enumerate}
        \item $S_{q-1}$, that is, the set of elements in $\F_{q^t}^*$ which are $(q-1)$-powers of elements of $W$;
        \item $S_{q+1}$, that is, the set of elements in $\F_{q^t}^*$ which are $(q+1)$-powers of elements of $W$;
        \item $T$, the set of elements $m\in\Fqt$ such that $\N_{q^t/q}(m)=1$.
    \end{enumerate}
    The equation $\N_{q^{t}/q}(x)=1$ has $(q^t-1)/(q-1)$ solutions.

    Assume that $t$ is even.
    Combining with Proposition~\ref{p:powers},
    \begin{equation}\label{e:206a}
    |S_{q-1}\cup S_{q+1}\cup T|\le\frac{q^t-1}{q-1}+\frac{q^t-1}{q+1}+\frac{q^t-1}{q-1}=\frac{(q^t-1)(3q+1)}{q^2-1},
    \end{equation}
    that for $q>3$ is less than $q^t-1$.

    In the case that $t$ is odd one obtains similarly
    \[
    |S_{q-1}\cup S_{q+1}\cup T|\le\frac{q^t-1}{q-1}+\frac{q^t-1}{2}+\frac{q^t-1}{q-1}=\frac{q^t-1}{2(q-1)}(q+3).
    \]
    For $q>5$ this is less than $q^t-1$.
\end{proof}

\section*{Acknowledgements}
The authors are very thankful to Giovanni Longobardi for pointing out Proposition \ref{prop:unpub} and to the anonymous referees for their comments.
The authors were partially supported by the Italian National Group for Algebraic and Geometric Structures and their Applications (GNSAGA - INdAM). The research of the first and the third author was supported by the INdAM - GNSAGA Project \emph{Tensors over finite fields and their applications}, number E53C23001670001. The research of the third author was supported by the project ``VALERE: VAnviteLli pEr la RicErca" of the University of Campania ``Luigi Vanvitelli'' and by the project COMBINE from ``VALERE: VAnviteLli pEr la RicErca" of the University of Campania ``Luigi Vanvitelli''.

Valentino Smaldore (ORCID: 0000-0003-3433-3164) and Corrado Zanella (ORCID: 0000-0002-5031-1961)\\
Dipartimento di Tecnica e Gestione dei Sistemi Industriali\\
Universit\`a degli Studi di Padova\\
Stradella S. Nicola, 3\\
36100 Vicenza VI - Italy\\
{\em \{valentino.smaldore,corrado.zanella\}@unipd.it}

\smallskip

Ferdinando Zullo (ORCID: 0000-0002-5087-2363)\\
Dipartimento di Matematica e Fisica,\\
Universit\`a degli Studi della Campania ``Luigi Vanvitelli'',\\
Viale Lincoln 5,\\
81100 Caserta CE - Italy\\
{{\em ferdinando.zullo@unicampania.it}}

\end{document}